\documentclass[11pt]{amsart}
 
\usepackage{amssymb,amsthm,amsmath}
\usepackage{comment}

\usepackage{color}
\usepackage{hyperref}
\hypersetup{ocgcolorlinks=true,allcolors=testc}
\hypersetup{
     colorlinks   = true,
     citecolor    = black
}
\hypersetup{linkcolor=black}
\hypersetup{urlcolor=black}

\newcommand{\dd}{\mathrm{d}}
\newcommand{\E}{\mathbb{E}}
\newcommand{\1}{\textbf{1}}
\newcommand{\R}{\mathbb{R}}
\newcommand{\vp}{\varphi}

\newcommand{\mc}[1]{\mathcal{{#1}}}

\usepackage[paper=a4paper, left=1.2in, right=1.2in, top=1.2in, bottom=1.2in]{geometry}
\usepackage{color}

\newtheorem{theorem}{Theorem}
\newtheorem{lemma}[theorem]{Lemma}
\newtheorem{corollary}[theorem]{Corollary}
\newtheorem{proposition}[theorem]{Proposition}

\theoremstyle{remark}
\newtheorem{remark}[theorem]{Remark}

\newtheorem{question}[theorem]{Question}

\begin{document}

\title{Sharp comparison of moments and the log-concave moment problem}

\author{Alexandros Eskenazis}
\address{(A.\!~E.) Mathematics Department\\ Princeton University\\ Fine Hall, Washington Road, Princeton, NJ 08544-1000, USA}
\email{ae3@math.princeton.edu}

\author{Piotr Nayar}
\address{(P.\!~N.) Institute of Mathematics\\ University of Warsaw\\ Banacha 2, 02-097, Warsaw, Poland}
\email{nayar@mimuw.edu.pl}

\author{Tomasz Tkocz}
\address{(T.\!~T.) Department of Mathematical Sciences \\ Carnegie Mellon University\\ Pittsburgh, PA 15213, USA}
\email{ttkocz@math.cmu.edu}

\thanks{This material is partially based upon work supported by the NSF grant DMS-1440140, while
the authors were in residence at the MSRI in Berkeley, California, during the fall semester of 2017. A.\!~E. and T.\!~T. were
also partially supported by the Simons Foundation and P.\!~N. by the ERC Starting Grant CONC-VIA-RIEMANN and by the National Science Centre Poland grant
2015/18/A/ST1/00553..}

\begin{abstract}This article investigates sharp comparison of moments for various classes of random variables appearing in a geometric context. In the first part of our work we find the optimal constants in the Khintchine inequality for random vectors uniformly distributed on the unit ball of the space $\ell_q^n$ for $q\in(2,\infty)$, complementing past works that treated $q\in(0,2]\cup\{\infty\}$. As a byproduct of this result, we prove an extremal property for weighted sums of symmetric uniform distributions among all symmetric unimodal distributions. In the second part we provide a one-to-one correspondence between vectors of moments of symmetric log-concave functions and two simple classes of piecewise log-affine functions. These functions are shown to be the unique extremisers of the $p$-th moment functional, under the constraint of a finite number of other moments being fixed, which is a refinement of the description of extremisers provided by the generalised localisation theorem of Fradelizi and Gu\'edon [Adv.\! Math.\! 204 (2006) no.\! 2, 509--529].
\end{abstract}

\maketitle

{\footnotesize

}


\section{Introduction}

This paper is devoted to two results concerning moments of log-concave random variables. The first is a sharp Khintchine-type inequality for linear functionals of random vectors uniformly distributed on the unit balls of $\ell_q^n$ for $q\in(2,\infty)$. The second is a precise description of sequences of moments of symmetric log-concave functions on the real line. The approach to both these results is based on the same simple idea which we shall now \mbox{briefly explain.}

Suppose we are given two real random variables $X,Y$ that satisfy $\E|X|^{p_i}=\E|Y|^{p_i}$ for $i=1,\ldots,n$, where $p_1,\ldots,p_n$ are distinct real numbers, and a function $\vp:\R \to \R$ for which we want to prove the inequality $\E \vp(X) \geq \E \vp(Y)$. Let $f_X, f_Y$ be the densities of $X$ and $Y$ respectively. We would like to show that $\int_\R \vp\cdot(f_X-f_Y) \geq 0$. Using the constraints, we see that this integral can be rewritten as
\begin{equation} \label{eq:generalstrategy}
	\int_\R \vp\cdot(f_X-f_Y) = \int_\R \Big( \vp(t)+ \sum_{i=1}^n c_i t^{p_i} \Big)\big(f_X(t)-f_Y(t)\big) \dd t,
\end{equation}
for every $c_1,\ldots,c_n\in\R$. Suppose additionally that $f_X-f_Y$ changes sign in exactly $n$ points $t_1,\ldots, t_n\in\R$. It turns out that it is always possible to choose the parameters $c_1,\ldots,c_n$ in such a way that the other factor $h(t)=\vp(t)+ \sum_{i=1}^n c_i t^{p_i}$ also vanishes in these points. Therefore, if one can prove (under additional assumptions on $\varphi$) that $h$ actually changes sign \emph{only} in $t_1,\ldots, t_n$, then the integrand in \eqref{eq:generalstrategy} has a fixed sign and the desired inequality follows.
  

\subsection{A sharp Khintchine-type inequality}\label{sec:intro-khintch}

Let $X = (X_1,\ldots,X_n)$ be a random vector in $\R^n$. A Khintchine inequality is a comparison of moments of linear forms $S=\sum_{i=1}^n a_iX_i$ of $X$, namely an inequality of the form $\|S\|_p \leq C_{p,q,X}\|S\|_q$, for $p,q>0$, where $\|S\|_r=(\E|S|^r)^{1/r}$ denotes the $r$-th moment of $S$. Here the constant $C_{p,q,X}$ depends only on $p,q$ and the distribution of $X$, but not on the vector $(a_1,\ldots,a_n)$. Since the second moment $\|S\|_2$ has an explicit expression in terms of the coefficients $a_1,\ldots,a_n$, the most commonly used Khintchine inequalities are of the form
\begin{equation}\label{eq:l2lp}\tag{$\star$}
A_{p,X} \Big\| \sum_{i=1}^n a_iX_i \Big\|_2 \leq \Big\|\sum_{i=1}^n a_i X_i\Big\|_p \leq B_{p,X} \Big\|\sum_{i=1}^n a_i X_i\Big\|_2.
\end{equation}
It is of interest to study the best constants $A_{p,X}$ and $B_{p,X}$ such that the above inequality holds for all real numbers $a_1,\ldots,a_n$.
In this setting, the classical Khintchine inequality (see \cite{Khi}) corresponds to a random vector $X$ uniformly distributed on the discrete cube $\{-1,1\}^n$. Then, one of the two sharp constants $A_{p,X}=A_{p,n}$ or $B_{p,X}=B_{p,n}$, depending on the value of $p$, is always equal to 1. To the best of our knowledge, the other optimal constant is known only for some ranges of $p$, namely for $p\geq3$ by the work \cite{Whi} of Whittle (see also \cite{Eat} and \cite{Kom}) and for $p\leq p_0\approx 1.8474$ by the works of Szarek \cite{Sz} and Haagerup \cite{Haa}. The asymptotically sharp constants $A_p = \inf_{n \geq 1} A_{p,n}$ and $B_p = \sup_{n\geq 1} B_{p,n}$ have been determined for all $p>0$ (see \cite{Haa}). We refer to \cite{LO}, \cite{KK}, \cite{BC}, \cite{NO} and \cite{Kon} for Khintchine inequalities for other random vectors.

In this article we consider random vectors $X$ uniformly distributed on the unit ball $B_q^n = \{x \in \R^n: \ |x_1|^q+\ldots+|x_n|^q \leq 1\}$ of the space $\ell_q^n$, where $q>0$. As usual, we denote by $B_\infty^n = [-1,1]^n$ the unit cube. We are interested in the values of the best constants $A_{p,X}=A_{p,q,n}$ and $B_{p,X}=B_{p,q,n}$ such that inequality \eqref{eq:l2lp} holds for all real numbers $a_1,\ldots,a_n$. In \cite{LO}, Lata\l a and Oleszkiewicz determined these constants for all $p \geq 1$ and $q=\infty$, that is, when $X_1,\ldots,X_n$ are i.i.d. random variables uniformly distributed on $[-1,1]$ (see also Section \ref{sec:remarks} for a short proof of their theorem). For $q<\infty$, the question was first raised by Barthe, Gu\'edon, Mendelson and Naor in \cite{BGMN}, who estimated the values of the optimal constants up to universal multiplicative factors for every $p, q \geq 1$. In the recent work \cite{ENT}, we found the sharp values of $A_{p,q,n}$ and $B_{p,q,n}$ for all $q \in (0,2]$ and $p > -1$ via a reduction to moments of \emph{Gaussian mixtures}, yet this approach fails for $q>2$. The first goal of this paper is to address the problem for the remaining range $q\in(2,\infty)$, when additionally we shall assume that $p\geq1$, thus answering Question 6 of \cite{ENT}.

As observed in \cite[Lemma~6]{BGMN}, if $X=(X_1,\ldots,X_n)$ is uniformly distributed on $B_q^n$ for some $q > 0$, then for every $p > -1$ and real numbers $a_1,\ldots,a_n$ we have
\begin{equation}\label{eq:Bpnmomentident}
\Big\|\sum_{i=1}^n a_iX_i\Big\|_p = \beta_{p,q,n}\Big\|\sum_{i=1}^n a_iY_i\Big\|_p,
\end{equation}
where $Y_1,\ldots,Y_n$ are i.i.d. random variables with density proportional to $e^{-|x|^q}$ and $\beta_{p,q,n}$ is a positive constant, given explicitly by
\begin{equation}\label{eq:beta}
\beta_{p,q,n} = \frac{\|X_1\|_p}{\|Y_1\|_p} = \left(\frac{\Gamma(n/q+1)}{\Gamma((n+p)/q+1)}\right)^{1/p}.
\end{equation}
This identity is a crucial observation which reduces finding the optimal constants in Khintchine's inequality for $X$ whose coordinates are dependent to $Y = (Y_1,\ldots,Y_n)$, which has i.i.d. components. Therefore, we restrict our attention to the latter case.

\begin{theorem}\label{thm:khintchexp-edge}
Fix $q \in [2,\infty]$ and $n \geq 1$. If $Y_1,\ldots,Y_n$ are i.i.d. random variables with density functions proportional to $e^{-|x|^q}$, then for every unit vector $(a_1,\ldots,a_n)$ and $p\geq2$ we have
\begin{equation}\label{eq:khintchexp-edge}
\|Y_1\|_p \leq \Big\| \sum_{i=1}^n a_i Y_i\Big\|_p,
\end{equation}
whereas for $p\in[1,2]$ the inequality is reversed. This is clearly sharp.
\end{theorem}

Denote by $\gamma_p = \sqrt{2} \Big(\frac{\Gamma (\frac{p+1}{2})}{\sqrt{\pi}} \Big)^{1/p}$ the $p$-th moment of a standard Gaussian random variable.

\begin{theorem}\label{thm:khintchexp-gauss}
Fix $q \in [2,\infty]$. If $Y_1, Y_2, \ldots$ are i.i.d. random variables with density functions proportional to $e^{-|x|^q}$, then for every $n\geq1$, real numbers $a_1,\ldots,a_n$ and $p\geq2$ we have
\begin{equation}\label{eq:khintchexp-gauss}
\Big\| \sum_{i=1}^n a_iY_i \Big\|_p \leq \gamma_p \Big\|\sum_{i=1}^n a_iY_i\Big\|_2,
\end{equation}
whereas for $p\in[1,2]$ the inequality is reversed. The above constant is optimal.
\end{theorem}

Combining Theorems \ref{thm:khintchexp-edge} and \ref{thm:khintchexp-gauss} with the crucial identity \eqref{eq:Bpnmomentident}, we get the following consequence for random vectors uniformly distributed on $B_q^n$.

\begin{corollary}\label{cor:Bpnfixedn}
Fix $q\in[2,\infty]$ and $n\geq1$. If $X=(X_1,\ldots,X_n)$ is a random vector uniformly distributed on $B_q^n$, then for every real numbers $a_1,\ldots,a_n$ and $p\geq2$ we have
\begin{equation}\label{eq:khintchineagain}
A_{p,q,n} \Big\| \sum_{i=1}^n a_iX_i \Big\|_2 \leq \Big\|\sum_{i=1}^n a_i X_i\Big\|_p \leq B_{p,q,n} \Big\|\sum_{i=1}^n a_i X_i\Big\|_2,
\end{equation}
where
\begin{equation}\label{eq:Bpnfixedn}
A_{p,q,n} = \begin{cases}\frac{\beta_{p,q,n}}{\beta_{2,q,n}}\gamma_p, & p\in[1,2) \\ \frac{\|X_1\|_p}{\|X_1\|_2}, & p\in[2,\infty) \end{cases} \ \ \ \mbox{and} \ \ \ B_{p,q,n}= \begin{cases} \frac{\|X_1\|_p}{\|X_1\|_2}, & p\in[1,2) \\ \frac{\beta_{p,q,n}}{\beta_{2,q,n}}\gamma_p, & p\in[2,\infty) \end{cases}.
\end{equation}
This value of $A_{p,q,n}$ is sharp for $p \in [2,\infty)$ and of $B_{p,q,n}$ for $p \in [1,2)$.
\end{corollary}

The infimal (respectively supremal) values of these constants $A_p$ (resp. $B_p$) over $n\geq1$ provide the answer to Question 6 of \cite{ENT}.

\begin{corollary}\label{cor:khintchBpn}
Fix $q \in [2,\infty]$. If $n\geq1$ and $X=(X_1,\ldots,X_n)$ is a random vector uniformly distributed on $B_q^n$, then for every real numbers $a_1,\ldots,a_n$ and $p\geq1$ we have
\begin{equation}\label{eq:khintchineagain_asymp}
A_p \Big\| \sum_{i=1}^n a_iX_i \Big\|_2 \leq \Big\|\sum_{i=1}^n a_i X_i\Big\|_p \leq B_p \Big\|\sum_{i=1}^n a_i X_i\Big\|_2,
\end{equation}
where
\begin{equation}\label{eq:khintchBpnconstants}
A_p = \begin{cases}\gamma_p, & p\in[1,2) \\ \frac{3^{1/2}}{(p+1)^{1/p}}, & p\in[2,\infty) \end{cases} \ \ \ \mbox{and} \ \ \ B_p= \begin{cases} \frac{3^{1/2}}{(p+1)^{1/p}}, & p\in[1,2) \\ \gamma_p, & p\in[2,\infty) \end{cases}
\end{equation}
The above constants are sharp.
\end{corollary}

It will be evident from the proof of Corollary \ref{cor:khintchBpn} that the dimension-dependent constants \eqref{eq:Bpnfixedn} improve upon the asymptotically sharp constants given in \eqref{eq:khintchBpnconstants}.

\begin{question}
Fix $q \in (2,\infty)$, $n\geq1$ and let $X=(X_1,\ldots,X_n)$ be a random vector uniformly distributed on $B_q^n$. For $p\geq2$ (respectively $p\in[1,2]$), what are the optimal values of $B_{p,q,n}$ (resp. $A_{p,q,n}$) in \eqref{eq:khintchineagain}? More ambitiously, which unit vectors $(a_1,\ldots,a_n)$ maximise (resp. minimise) the moments $\big\|\sum_{i=1}^n a_iX_i\big\|_p$?
\end{question}

Our arguments rely on the convexity of certain functions and work in fact for the whole range $q > 0$. However, when $p < 1$ those functions are no longer convex. The technique developed in \cite{ENT} for the range $q \in (0,2]$ has the advantage of covering all $p > -1$. It remains an open problem to understand the optimal constants for $q > 2$ and $p \in (-1,1)$.

\subsection{An extremal property of symmetric uniform distributions}
Before proceeding to the second main part of the present article, we mention an extremal property of symmetric uniform random variables which was motivated by a similar property of independent symmetric random signs $\varepsilon_1, \varepsilon_2, \ldots$. In \cite{FHJSZ} and independently in \cite{Pin}, the authors showed that an Orlicz function $\Phi:\R\to\R$ of class $C^2$ satisfies the inequality
\begin{equation}
\E \Phi\Big(\sum_{i=1}^n X_i\Big) \geq \E\Phi\Big(\sum_{i=1}^n \sigma_i \varepsilon_i\Big),
\end{equation}
for every symmetric independent random variables $X_1,X_2,\ldots$ and real numbers $\sigma_1,\sigma_2,\ldots$ such that $\sigma_i^2 = \E X_i^2$ if and only if $\Phi''$ is convex on $\R$. This result, when applied to $\Phi(x) = |x|^p$ and $X_i$ being standard Gaussian random variables allows one to derive the optimal constants in the classical Khintchine inequality for $p \geq 3$. For $p \in (0,3)$ all available proofs (see \cite{Haa}, \cite{NP}, \cite{Mor}) are subtle and more technical. We obtain the following analogue of the above theorem for symmetric unimodal random variables, i.e. continuous random variables whose densities are even and nonincreasing on $[0,\infty)$.

\begin{theorem}\label{thm:Phi}
Let $U_1,U_2,\ldots$ be independent random variables uniformly distributed on $[-\sqrt{3},\sqrt{3}]$, thus having $\E U_i^2=1$. An even function $\Phi:\R\to\R$ of class $C^3$ satisfies
\begin{equation}\label{eq:Phi}
\E \Phi\Big(\sum_{i=1}^n X_i\Big) \geq \E\Phi\Big(\sum_{i=1}^n \sigma_i U_i\Big),
\end{equation}
for every symmetric unimodal independent random variables $X_1,X_2,\ldots$ and real numbers $\sigma_1,\sigma_2,\ldots$, such that $\sigma_i^2 = \E X_i^2$, if and only if $\Phi'''(x) \geq 0$ for every $x \geq 0$. Moreover, the reverse inequality holds if and only if $\Phi'''(x) \leq 0$ for every $x \geq 0$.
\end{theorem}

As we will explain in Remark \ref{rem:regularity}, thanks to the convexity of the function $|x|^p$ for $p \geq 1$, this theorem allows us to recover the sharp Khintchine inequality for symmetric uniform random variables of \cite{LO} for all $p \geq 1$ (see also Proposition \ref{prop:p<1}).



\subsection{The log-concave moment problem} Recall that a function $f:\R^n\to\R_+$ is called log-concave if $f=e^{-V}$ for a convex function $V:\R^n \to\R\cup\{\infty\}$. 
For a symmetric log-concave function $f:\R\to\R_+$ and $p_1,\ldots,p_{n+1}\in(-1,\infty)$, consider the moment functionals $m_i(f)=\int_0^\infty t^{p_i}f(t)\dd t$, $i=1,\ldots,n+1$. For simplicity, we shall restrict our attention to the class $\mc{L}$ of symmetric log-concave functions that additionally satisfy $f(0)=1$. Our goal is to describe all possible sequences $(m_1,\ldots, m_n)$ arising as moment sequences of functions $f \in \mc{L}$, i.e. $m_i=m_i(f)$ for $i=1,\ldots,n$. For $k\geq0$, consider the following classes of \emph{simple} log-concave functions,
\begin{equation} \label{eq:Lndefine}
\begin{split}
\mc{L}_{2k}^-&=\big\{f(t)=\exp\big(-a_1|t|-a_2(|t|-b_2)_+-\cdots-a_k(|t|-b_k)_+ \big){\bf 1}_{|t|\leq b_{k+1}}\big\}\\
\mc{L}_{2k+1}^+ &= \big\{f(t)=\exp\big(-a_1|t|-a_2(|t|-b_2)_+-\cdots-a_{k+1}(|t|-b_{k+1})_+\big\}\\
\mc{L}_{2k}^+&=\big\{f(t)=\exp\big( -a_1(|t|-b_1)_+-\cdots-a_k(|t|-b_k)_+\big)\big\}\\
\mc{L}_{2k+1}^-&=\big\{f(t)=\exp\big( -a_1(|t|-b_1)_+-\cdots-a_k(|t|-b_k)_+\big){\bf 1}_{|t| \leq b_{k+1}}\big\},
\end{split}
\end{equation}
where the parameters satisfy $a_1, a_2, \ldots \in[0,\infty]$ and $0 \leq b_1 \leq b_2 \leq \ldots\leq\infty$. Here and throughout we will adopt the convention that $\infty \cdot 0=0$. We also set $\mc{L}_0^\pm = \{\1_{\{0\}},1\}$. For $n\geq0$, the space of parameters $(a,b)$ corresponding to $\mc{L}_n^\pm$ will be denoted by $\mc{P}_n^\pm$.
Notice that each $\mc{L}_n^{\pm}$ is an $n$-parameter family of functions. Moreover, these families form a hierarchical structure, namely it is not hard to check that
\begin{equation}\label{hierarchy}
\mc{L}_{n-1}^+ \cup \mc{L}_{n-1}^- = \mc{L}_{n}^+ \cap \mc{L}_{n}^- .
\end{equation}
It turns out that all possible moment sequences $(m_1,\ldots,m_n)$ arise as moment sequences of members of $\mc{L}_n^\pm$. To be more precise, we show the following theorem.

\begin{theorem} \label{moments-main}
Let $n\geq1$, $f\in\mc{L}$ and let $p_1,\ldots,p_{n+1}\in(-1,\infty)$ be distinct.
\begin{enumerate}
\item[\emph{(i)}] There exist unique functions $f_+ \in \mc{L}_n^+$ and $f_- \in \mc{L}_n^-$ such that \begin{equation} 
m_i(f)=m_i(f_+)=m_i(f_-), \ \ \  \mbox{for every } \ i=1,\ldots,n.
\end{equation}
\item[\emph{(ii)}] Let $p_{i_1}<p_{i_2}<\ldots<p_{i_{n+1}}$ be the increasing rearrangement of $p_1,\ldots,p_{n+1}$. If $n+1=i_k$ and $n+1-k$ is even, then 
\begin{equation} 
m_{n+1}(f_-) \leq m_{n+1}(f) \leq m_{n+1}(f_+).
\end{equation}
If $n+1-k$ is odd, then the above inequalities are reversed. Moreover, equality holds only if $f=f_+$ or $f=f_-$ respectively.

\end{enumerate}     
\end{theorem}

The above theorem should be compared with the work of Fradelizi and Gu\'edon on extremizing convex functionals under linear constraints, see \cite[Theorem~2]{FG}. There, the authors work with the class $\mc{L}_M$ of all log-concave functions $f$ supported on the bounded interval $[0,M]$, where $M\in(0,\infty)$. According to their theorem, among all log-concave functions $f \in \mc{L}_M$ with fixed values of $m_i(f)$, $i=1,\ldots,n$, the ones which maximise (or minimise) $m_{n+1}(f)$ have to be of a specific form, namely $f=e^{-V}$, where $V$ is a piecewise linear, convex function with at most $n$ linear pieces. In fact, a similar statement is true in much greater generality, that is, when one maximises a convex functional over a set of log-concave functions $f \in \mc{L}_M$ satisfying arbitrary linear constraints. Nevertheless, a log-concave function $f=e^{-V}\in\mc{L}_M$ with $V$ being piecewise linear with at most $n$ linear pieces is determined by $2n$ parameters, namely the slopes of these linear pieces, the $n-1$ points where these slopes (potentially) change and the value of $f(0)$. In contrast to that, the classes of simple functions $\mc{L}_n^+$ and $\mc{L}_n^-$ appearing in Theorem \ref{moments-main} depend on $n$ free parameters each and are in one-to-one correspondence with sequences of moments $(m_1,\ldots,m_n)$. Theorem \ref{A=B=C} in Section \ref{sec:momentcomparison} provides further insight into the structure of the set of moment sequences of symmetric log-concave functions.

The rest of this paper is organised as follows. In Section \ref{sec:proofkhintch} we present the proofs of Theorems \ref{thm:khintchexp-edge} and \ref{thm:khintchexp-gauss} along with the derivation of Corollaries \ref{cor:Bpnfixedn} and \ref{cor:khintchBpn}. The proof of Theorem \ref{thm:Phi} and some related remarks appear in Section \ref{sec:remarks}.  Finally, Section \ref{sec:momentcomparison} contains the proof of Theorem \ref{moments-main}.

\subsection*{Acknowledgments}

We are indebted to Olivier Gu\'edon for his great help with a preliminary version of this manuscript, valuable feedback and constant encouragement. We are also grateful to Rafa{\l}  Lata\l a for a stimulating discussion. Piotr Nayar would like to thank Bo'az Klartag for his kind hospitality at the Weizmann Institute of Science in August 2017. The accommodation during this visit has been provided from the ERC Starting Grant DIMENSION.


\section{Sharp Khintchine inequalities on $B_q^n$}\label{sec:proofkhintch}

We start by proving Theorems \ref{thm:khintchexp-edge} and \ref{thm:khintchexp-gauss}. Let $Y_1^{(q)}, Y_2^{(q)},\ldots$ be i.i.d. random variables with density $f_q(x) = c_qe^{-|x|^q}$, where $c_q=\big(2\Gamma(1+1/q)\big)^{-1}$ is the normalising constant and $q\in[2,\infty)$. For $p\in(0,\infty)$, consider also the normalised random variables
\begin{equation}
Y_{i,p}^{(q)} = Y_i^{(q)}/\|Y_i^{(q)}\|_p.
\end{equation}
The essence of our argument comprises two main parts. First, we show that the densities $f_q$ interlace well as $q$ varies which gives the monotonicity of $q \mapsto \E h(Y_{i,p}^{(q)})$ for certain test functions $h$ and every $p\in(0,\infty)$ (see Lemma \ref{lm:singlecomparisonp}; the same idea was used for instance in \cite{KK} and \cite{BN}). Afterwards, combining this with an inductive procedure gives the monotonicity of moments of $S=\sum_{i=1}^n a_i Y_i^{(q)}$ with respect to $q$. Finally, comparing against Gaussian random variables, which correspond to $q=2$, gives the desired results.

We remark that each $Y_i^{(q)}$ is a symmetric unimodal random variable, that is, a continuous random variable whose density is even and nonincreasing on $[0,\infty)$. We shall need two basic facts about symmetric unimodal random variables (for the proofs see, for instance, Lemmas 1 and 2 in \cite{LO}).

\begin{lemma}\label{lm:unimodmix}
A symmetric unimodal random variable is of the form $R\cdot U$, where $R$ is a positive random variable and $U$ is uniformly distributed on $[-1,1]$, independent of $R$.
\end{lemma}

\begin{lemma}\label{lm:unimodsum}
A sum of independent symmetric unimodal random variables is a symmetric unimodal random variable.
\end{lemma}

In other words, symmetric unimodal random variables are mixtures of uniform random variables and the convolution of even and unimodal densities is even and unimodal. (Note that analogous properties are also true for symmetric random variables.)

\subsection{Proof of Theorems \ref{thm:khintchexp-edge} and \ref{thm:khintchexp-gauss}}
The main result of this section is the following monotonicity statement, which implies Theorems \ref{thm:khintchexp-edge} and \ref{thm:khintchexp-gauss}.

\begin{theorem}\label{thm:monot_single}
Let $a_1,\ldots,a_n$ be real numbers. 
\begin{enumerate}
\item[\emph{(i)}] The function $\psi_1: (0,\infty)^n \to \R$ given by
\begin{equation}
	\psi_1(q_1,\ldots,q_n) = \E\Big|\sum_{i=1}^n a_iY_{i,p}^{(q_i)}\Big|^p
\end{equation}
is coordinatewise nondecreasing when $p \geq 2$ and nonincreasing when $p\in[1,2]$.
\item[\emph{(ii)}] The function $\psi_2:(0,\infty)^n \to \R$ given by
\begin{equation}
\psi_2(q_1,\ldots,q_n) = \E\Big|\sum_{i=1}^n a_i Y_{i,2}^{(q_i)}\Big|^p
\end{equation}
is coordinatewise nonincreasing when $p \geq 2$ and nondecreasing when $p\in[1,2]$.
\end{enumerate}
\end{theorem}

A fact similar to the monotonicity of $\psi_1$ for slightly different random variables has been established in the case $p=1$ in \cite{BN} (see the proof of Theorem 3 therein). We first show that this theorem implies Theorem \ref{thm:khintchexp-edge} and \ref{thm:khintchexp-gauss}.

\begin{proof}[Proof of Theorem \ref{thm:khintchexp-edge} and \ref{thm:khintchexp-gauss}] 
To prove Theorem \ref{thm:khintchexp-edge}, note that $Y_{i,p}^{(2)}$ are i.i.d. centred Gaussian random variables with $p$-th moment equal to one.  By Theorem \ref{thm:monot_single}{\it (i)}, for every unit vector $(a_1,\ldots,a_n)$ and $p\geq2$, we have
\begin{align*}
\Big\|\sum_{i=1}^n a_i\frac{Y_i^{(q)}}{\|Y_i^{(q)}\|_p} \Big\|_p &= \Big\|\sum_{i=1}^n a_iY_{i,p}^{(q)} \Big\|_p 
\geq 
\Big\| \sum_{i=1}^n a_iY_{i,p}^{(2)} \Big\|_p 
= \left\| \Big(\sum_{i=1}^n a_i^2\Big)^{1/2} Y_{1,p}^{(2)} \right\|_p = 1,
\end{align*}
which immediately yields \eqref{eq:khintchexp-edge}. When $p \in [1,2]$, the above estimate gets reversed. 

To get Theorem \ref{thm:khintchexp-gauss}, note that $Y_{i,2}^{(2)}$ are i.i.d. standard Gaussian random variables.  By Theorem \ref{thm:monot_single}\emph{(ii)}, for every real numbers $a_1,\ldots,a_n$ and $p\geq2$, we have
\begin{align*}
\Big\| \sum_{i=1}^n a_i\frac{Y_i^{(q)}}{\|Y_i^{(q)}\|_2} \Big\|_p &= \Big\|\sum_{i=1}^n a_iY_{i,2}^{(q)} \Big\|_p 
\leq 
\Big\| \sum_{i=1}^n a_iY_{i,2}^{(2)} \Big\|_p 
= \gamma_p\Big\|\sum_{i=1}^n a_i\frac{Y_i^{(q)}}{\|Y_{i}^{(q)}\|_2} \Big\|_2
\end{align*}
which immediately yields \eqref{eq:khintchexp-gauss}. When $p \in [1,2]$, the above estimate gets reversed. The constant $\gamma_p$ is sharp by the Central Limit Theorem.
\end{proof}

The proof of Theorem \ref{thm:monot_single} relies on the following lemma.

\begin{lemma}\label{lm:singlecomparisonp}
Let $0 <  q < r$ and $p\in(0,\infty)$. For every convex function $h:[0,+\infty)\to\R$ we have
\begin{equation}
\E h(|Y_{1,p}^{(q)}|^p) \geq \E h(|Y_{1,p}^{(r)}|^p).
\end{equation}
\end{lemma}
\begin{proof}
Let $\phi_q$ be the density of $Y_{1,p}^{(q)}$. By symmetry, the assertion is equivalent to
\begin{equation} \label{eq:compgoal}
\int_0^\infty h(x^p)\big(\phi_q(x)-\phi_r(x)\big) \dd x \geq 0.
\end{equation}
Since the density $\phi_s(x)$ is of the form $b_se^{-a_s|x|^s}$, there is an interval $(A,B) \subset (0,\infty)$ such that the difference $\phi_q(x) - \phi_r(x)$ is negative on $(A,B)$ and positive on $(0,A)\cup(B,\infty)$. Indeed, it is clear that the graphs of $\phi_q$ and $\phi_r$ have to intersect on $(0,+\infty)$ at least twice because both functions are probability densities with the same $p$-th moments (see also Lemma \ref{l3}). On the other hand, by the convexity of $x\mapsto\ln\big(\phi_q(x^{1/q})/\phi_r(x^{1/q})\big)$ one can easily check that they cannot intersect more than twice. Finally, $\phi_q(x) - \phi_r(x)$ is plainly positive for $x$ large enough, since $q<r$.

Choose $\alpha$ and $\beta$ such that $h(x^p) - \alpha x^p - \beta$ vanishes at $x = A$ and $x = B$. Since $h$ is convex, the function $h(x^p) - \alpha x^p - \beta$ is nonpositive on $(A,B)$ and nonnegative on $(0,A)\cup (B,\infty)$. Therefore,
\[
\big(h(x^p) - \alpha x^p - \beta\big)\big(\phi_q(x) - \phi_r(x)\big) \geq 0\]
for every $x > 0$ and integrating yields the desired inequality \eqref{eq:compgoal}.
\end{proof}

To derive Theorem \ref{thm:monot_single} from Lemma \ref{lm:singlecomparisonp}, we shall establish the convexity of certain functions $h$, which is settled by the following elementary lemma.  

\begin{lemma}\label{lm:conv1/p}
\begin{itemize}
\item[\it (i)] The function $h_1(x) = |x^{1/p}+1|^p+|x^{1/p}-1|^p$, $x\geq 0$ is convex for $p\in[1,2]$ and concave for $p \geq 2$.
\item[\it (ii)] The function $h_2(x) =  \int_{-1}^{1}|x^{1/2}+u|^p\dd u$, $x \geq 0$ is concave for $p\in[1,2]$ and convex for $p \geq 2$.
\end{itemize}
\end{lemma}
\begin{proof}
{\it (i)} For $y \ne 1$ we have $h_1'(y^{-p})=|1+y|^{p-1}+\textrm{sgn}(1-y)|1-y|^{p-1}$ and therefore
\[
	-p(p-2)y^{-p-1} h_1''(y^{-p}) = (p-2)(h_1'(y^{-p}))' = (p-1)(p-2)[|1+y|^{p-2}-|1-y|^{p-2}]\geq0,
\] 
for all values of $p\geq1$ and $y\geq0$.

\noindent \emph{(ii)} We have
\[
\frac{\dd}{\dd x}\int_{-1}^{1}|x^{1/2}+u|^p\dd u = \frac{\dd}{\dd x}\int_{x^{1/2}-1}^{x^{1/2}+1}|u|^p\dd u = \frac12  (|x^{1/2}+1|^p-|x^{1/2}-1|^p)/ x^{1/2},\]
so our goal is to show that the function $\vp_1(y)=\vp_2(y)/y$, $y \geq 0$, where $\vp_2(y)=|y+1|^p-|y-1|^p$, is monotone. Since $\vp_2(0)=0$, it suffices to observe that for $y \ne 1$
\[
(p-2)\vp_2''(y)=p(p-1)(p-2)(|y+1|^{p-2}-|y-1|^{p-2})\geq0,
\] 
for all values of $p\geq1$ and $y\geq0$, and then use the monotonicity of slopes of the function $(p-2)\vp_2$.
\end{proof} 

\begin{proof}[Proof of Theorem \ref{thm:monot_single}]
It clearly suffices to show the desired monotonicity with respect to $q_1$. To prove monotonicity of $\psi_1$ let us define $S = \sum_{i=2}^na_iY_{i,p}^{(q_i)}$. By symmetry we have
\[
\E\Big|\sum_{i=1}^n a_iY_{i,p}^{(q_i)}\Big|^p = \E\big|a_1Y_{1,p}^{(q_1)}+S\big|^p = \E\big|a_1|Y_{1,p}^{(q_1)}|+S\big|^p = \E_Yh(|Y_{1,q_1}|^p),\]
where, again by symmetry of $S$, we have
\begin{equation}
h(x) = \E_S|a_1x^{1/p}+S|^p =\frac{1}{2}\E_S\Big[|a_1x^{1/p}+S|^p+|a_1x^{1/p}-S|^p\Big],
\end{equation}
which, by virtue of Lemma \ref{lm:conv1/p}, is an average of convex functions when $p\leq2$ (respectively concave when $p\geq 2$). The conclusion follows from Lemma \ref{lm:singlecomparisonp}.

To prove the claim for $\psi_2$ , let $S = \sum_{i=2}^na_iY_{i,2}^{(q_i)}$. By symmetry we can write
\[
\E\Big|\sum_{i=1}^n a_iY_{i,2}^{(q_i)}\Big|^p = \E\big|a_1Y_{1,2}^{(q_1)}+S\big|^p = \E\big|a_1|Y_{1,2}^{(q_1)}|+S\big|^p = \E_Yh(|Y_{1,2}^{(q_1)}|^2),\]
where $h(x) = \E_S|a_1\sqrt{x}+S|^p$. From Lemma \ref{lm:unimodsum}, $S$ is symmetric and unimodal and thus $S$ has the same distribution as $RU$, where $R$ is a positive random variable and $U$ is a uniform random variable on $[-1,1]$, independent of $R$. We therefore have
\begin{equation}
h(x) = \E_R\left[ \frac{1}{2}\int_{-1}^1 |a_1\sqrt{x}+Ru|^p\dd u\right],\end{equation}
for some positive random variable $R$. By virtue of Lemma \ref{lm:conv1/p} this is an average of convex functions when $p\geq2$ (respectively concave when $p\leq 2$) and the conclusion follows from Lemma \ref{lm:singlecomparisonp} with $p=2$.
\end{proof}

\begin{remark}
The unimodality of $Y_i$ is essential for the monotonicity of $\psi_2$ and the derivation of the Gaussian constant $\gamma_p$ in the preceeding argument. In \cite{BN}, Barthe and Naor were interested in determining the optimal constants in the Khintchine inequality (with $p=1$) for a different family of random variables indexed by $q\in[1,\infty)$. Even though the exact analogue of Lemma \ref{lm:singlecomparisonp} was valid in their context as well, the lack of unimodality of those distributions when $q\in[1,2)$ makes an inductive argument as in the proof of Theorem \ref{thm:monot_single} fail and, in fact, the optimal constant for $q=1$ differs from $\gamma_1$ (see \cite{Sz}).
\end{remark}


\subsection{Constants in the Khintchine inequality}

A standard argument leads to Corollary \ref{cor:Bpnfixedn}. We include it for completeness. 


\medskip

\noindent {\it Proof of Corollary \ref{cor:Bpnfixedn}.}
Let $a_1,\ldots,a_n\in\R$ and $p\geq2$. The crucial identity \eqref{eq:Bpnmomentident} implies that \eqref{eq:khintchexp-edge} also holds for a random vector $X=(X_1,\ldots,X_n)$, uniformly distributed on $B_q^n$. Therefore, by homogeneity we get
$$\Big\|\sum_{i=1}^n a_i X_i\Big\|_p \geq \Big(\sum_{i=1}^n a_i^2\Big)^{1/2} \|X_1\|_p = \frac{\|X_1\|_p}{\|X_1\|_2} \Big\|\sum_{i=1}^n a_iX_i\Big\|_2.$$
For the reverse inequality, consider i.i.d. random variables $Y_1,\ldots,Y_n$ with density proportional to $e^{-|x|^q}$. Combining \eqref{eq:Bpnmomentident} and \eqref{eq:khintchexp-gauss}, we deduce that
$$\Big\|\sum_{i=1}^n a_i X_i\Big\|_p = \beta_{p,q,n}\Big\|\sum_{i=1}^n a_i Y_i\Big\|_p \leq \beta_{p,q,n}\gamma_p \Big\|\sum_{i=1}^n a_i Y_i\Big\|_2 = \frac{\beta_{p,q,n}\gamma_p}{\beta_{2,q,n}} \Big\|\sum_{i=1}^n a_iX_i\Big\|_2,$$
which completes the proof of \eqref{eq:Bpnfixedn} for $p\geq2$. The case $p\in[1,2]$ is identical.
\hfill$\Box$

\medskip

Given Corollary \ref{cor:Bpnfixedn}, deriving the constants in Corollary \ref{cor:khintchBpn} is now straightforward, but requires a bit of technical work.

\medskip

\noindent {\it Proof of Corollary \ref{cor:khintchBpn}.} For $n\geq1$, $p\geq2$ and real numbers $a_1,\ldots,a_n$ by Corollary \ref{cor:Bpnfixedn} we get
\begin{equation} \label{eq:usecor3}
\frac{\|X_1\|_p}{\|X_1\|_2}\Big\|\sum_{i=1}^na_iX_i\Big\|_2\leq\Big\| \sum_{i=1}^n a_i X_i\Big\|_p \leq \frac{\beta_{p,q,n}}{\beta_{2,q,n}} \gamma_p \Big\|\sum_{i=1}^n a_iX_i\Big\|_2.
\end{equation}
The optimal values \eqref{eq:khintchBpnconstants} of the constants $A_p, B_p$ in the Khintchine inequality \eqref{eq:khintchineagain} will easily follow from the following claim.

\smallskip

\noindent {\it Claim.} Suppose that $X^{(n)}=\big(X_1^{(n)},\ldots,X_n^{(n)}\big)$ is a random vector, uniformly distributed on $B_q^n$ for some $q\in(2,\infty)$. Then, the sequence $\big\{ \|X_1^{(n)}\|_p/\|X_1^{(n)}\|_2\big\}_{n=1}^\infty$ is nondecreasing.

\smallskip

Assume for now that the claim is true. By the crucial identity \eqref{eq:Bpnmomentident}, the sequences $\big\{ \|X_1^{(n)}\|_p/\|X_1^{(n)}\|_2\big\}_{n=1}^\infty$ and $\big\{\beta_{p,q,n}/\beta_{2,q,n}\big\}_{n=1}^\infty$ are proportional, so by the claim the latter is also nondecreasing. Thus, for every $n\geq1$, $p\geq2$ and real numbers $a_1,\ldots,a_n$, \eqref{eq:usecor3} yields that
$$A_p \Big\|\sum_{i=1}^n a_i X_i\Big\|_2 \leq \Big\|\sum_{i=1}^n a_iX_i\Big\|_p \leq B_p \Big\|\sum_{i=1}^n a_i X_i\Big\|_2,$$
where
$$A_p = \inf_{n\geq1} \frac{\|X_1^{(n)}\|_p}{\|X_1^{(n)}\|_2} = \frac{\|X_1^{(1)}\|_p}{\|X_1^{(1)}\|_2} = \frac{3^{1/2}}{(p+1)^{1/p}}$$
and
$$B_p = \gamma_p\cdot\sup_{n\geq1}\frac{\beta_{p,q,n}}{\beta_{2,q,n}} = \gamma_p\cdot\lim_{n\to\infty}\frac{\beta_{p,q,n}}{\beta_{2,q,n}} = \gamma_p,$$
as can be checked using \eqref{eq:beta} and Stirling's formula. The optimality of these constants follows from the sharpness of Theorems \ref{thm:khintchexp-edge} and \ref{thm:khintchexp-gauss}. 
 The proof for $p\in[1,2]$ works with the obvious adaptations.
\hfill$\Box$

\smallskip

\noindent {\it Proof of the claim.} Fix $p,q\geq2$ and for every $n\geq1$, denote $Y_n = X_1^{(n)}/\big\|X_1^{(n)}\big\|_2$. The $Y_n$ are symmetric unimodal random variables with densities of the form $f_n(x)=c_n (M_n-|x|^q)^{\frac{n-1}{q}}_+$ and an argument identical to the one used in Lemma \ref{lm:singlecomparisonp} shows that the graphs of $f_n$ and $f_{n+1}$ intersect exactly twice on $(0,\infty)$. Therefore, to prove that $\|Y_n\|_p\leq\|Y_{n+1}\|_p$, it suffices to prove that the sign pattern of $f_n-f_{n+1}$ is $(-,+,-)$ or, equivalently, that $M_n<M_{n+1}$. An elementary computation involving the beta function shows that
\begin{equation}
M_n = \frac{1}{\|X_1^{(n)}\|_2^q} = \left(\frac{\Gamma\big(\frac{1}{q}\big)}{\Gamma\big(\frac{3}{q}\big)}\cdot\frac{\Gamma\big(\frac{n}{q}+1+\frac{2}{q}\big)}{\Gamma\big(\frac{n}{q}+1\big)}\right)^{q/2},
\end{equation}
thus the proof will be complete once we prove that the function
$$\rho(x) = \frac{\Gamma(x+s)}{\Gamma(x+1)}, \ \ \ x\in(0,\infty)$$
is strictly increasing for $s=1+\frac{2}{q}>1$. It is well known that $\eta(x)=\log\Gamma(x)$ is strictly convex on $(0,\infty)$, hence
$(\log\rho)'(x) = \eta'(x+s)-\eta'(x+1) >0,$
since $s>1$, and the claim follows.
\hfill$\Box$


\section{Further remarks on uniform random variables}\label{sec:remarks}

The technique used to prove Theorems \ref{thm:khintchexp-edge} and \ref{thm:khintchexp-gauss} also provides a new proof of the result of Lata\l a and Oleszkiewicz from \cite{LO} which we shall now present. Fix $n\geq1$ and let $U_1,\ldots,U_n$ be independent random variables uniformly distributed on $[-1,1]$. The main result of \cite{LO} is that $\big\|\sum_{i=1}^n a_i U_i\Big\|_p$ as a function of $(a_1^2,\ldots,a_n^2)$ is Schur convex (resp. concave) for every $1 \leq p \leq 2$ (resp. $p \geq 2$). See \cite{MO} for further background on the Schur ordering. In particular, when, say, $p \geq 2$, the $p$-th moment $\Big\|\sum_{i=1}^n a_i U_i\Big\|_p$, as $(a_1,\ldots,a_n)$ varies over all unit vectors, is maximised for $a = (1/\sqrt{n},\ldots,1/\sqrt{n})$ and minimised for $a = (1,0,\ldots,0)$.

For $\lambda\in[0,1]$, let $X_\lambda=\sqrt{\lambda}U_1+\sqrt{1-\lambda}U_2$. The crux of the argument presented in \cite{LO} is the fact that for every symmetric unimodal random variable $V$ independent of the $U_i$ we have
\begin{equation} \label{eq:xlambda}
\E|X_\lambda+V|^p \leq \E|X_{\lambda'}+V|^p, \quad 0<\lambda<\lambda'<1/2,
\end{equation}
for $p \geq 2$ and the reverse for $1 \leq p \leq 2$. Then, Schur convexity follows by a standard argument based on Muirhead's lemma (see \cite[Lemma~B.1]{MO}). We shall sketch a different proof of this inequality, based on the idea of ``well intersecting'' densities described in the introduction and used in the proof of Lemma \ref{lm:singlecomparisonp}.

\smallskip

\noindent {\it A new proof of \eqref{eq:xlambda}.} Let $f_\lambda$ be the density of $X_\lambda$ and $h(x) = \E_V|\sqrt{x}+V|^p$. Since $V$ is a mixture of uniform random variables (Lemma \ref{lm:unimodmix}), it follows from Lemma \ref{lm:conv1/p} that this function is convex for $p \geq 2$ and concave for $1 \leq p \leq 2$. By symmetry, $$\E h(X_\lambda^2) = \E||X_\lambda|+V|^p = \E|X_\lambda+V|^p,$$ thus we want to show that $\E h(X_\lambda^2)\leq \E h(X_{\lambda'}^2)$ or, equivalently, that
\begin{equation} \label{eq:goallatole}
\int_0^\infty h(x^2)\big(f_{\lambda'}(x)-f_{\lambda}(x)\big)\dd x \geq 0.
\end{equation}
Since $\E X_\lambda^2 = \E U_1^2$ does not depend on $\lambda$, we can modify $h(x^2)$ in the integrand by any function of the form $\alpha x^2 + \beta$, writing
\[
\int_0^\infty h(x^2)\big(f_{\lambda'}(x)-f_{\lambda}(x)\big)\dd x = \int_0^\infty \big(h(x^2)-\alpha x^2-\beta\big)\cdot\big(f_{\lambda'}(x)-f_{\lambda}(x)\big)\dd x
\] 
The only technical part of the argument is to check that $f_{\lambda'}-f_\lambda$ changes sign exactly twice on $(0,\infty)$, say at $0 < A < B$ and that it is positive on $(0,A)$, negative on $(A,B)$ and nonnegative on $(B,\infty)$, yet this is elementary to check since both densities are trapezoidal with the same second moment. Having this, we finish as in the proof of Lemma \ref{lm:singlecomparisonp}: we choose $\alpha$ and $\beta$ to match the sign changes of the function $h(x^2) - \alpha x^2 - \beta$, so that the integrand is nonnegative and \eqref{eq:goallatole} follows.
\hfill$\Box$

\smallskip

We remark that for uniform random variables, both the approach from \cite{LO} and the one presented here break down for $p\in(-1,1)$. This is because the functions appearing in Lemma \ref{lm:conv1/p} fail to be convex or concave when $p$ is in this range. Nevertheless, uniform random variables satisfy the conclusion of Theorem \ref{thm:khintchexp-edge} for $p\in(-1,1)$, as shown by the following simple argument.

\begin{proposition} \label{prop:p<1}
Fix $p\in(-1,2)$ and $n\geq1$. If $U_1,\ldots,U_n$ are i.i.d. symmetric uniform random variables, then for every unit vector $(a_1,\ldots,a_n)$ we have
\begin{equation} \label{eq:p<1}
\Big\|\sum_{i=1}^n a_iU_i\Big\|_p \leq \|U_1\|_p.
\end{equation}
\end{proposition}

\begin{proof}
By Lemmas \ref{lm:unimodmix} and \ref{lm:unimodsum}, there exists a positive random variable $R$ such that $\sum_{i=1}^n a_i U_i$ has the same distribution as $RU_1$. Since $\|U_1\|_2=\|\sum_{i=1}^n a_iU_i\|_2 = \|R\|_2\|U_1\|_2$, we have $\|R\|_2=1$. Therefore, for $p\in(-1,2)$ we have
$$\Big\|\sum_{i=1}^n a_iU_i\Big\|_p = \|RU_1\|_p = \|R\|_p \|U_1\|_p\leq \|R\|_2 \|U_1\|_p = \|U_1\|_p,$$
which completes the proof.
\end{proof}

It is evident from the proof of Proposition \ref{prop:p<1} that an analogue of Lemma \ref{lm:unimodsum} about sums of random variables with density proportional to $e^{-|x|^q}$, $q\in(2,\infty)$, instead of uniforms would extend Theorem \ref{thm:khintchexp-edge} to all $p\in(-1,\infty)$ and $q\in(2,\infty]$. We refer to \cite{MOU} for more on distributions having this property.

We conclude this section with the proof of Theorem \ref{thm:Phi}. According to Lemma \ref{lm:unimodmix}, inequality \eqref{eq:Phi} is equivalent to the validity of
\begin{equation}\label{eq:with-R}
	\E \Phi\Big( \sum_{i=1}^n R_i U_i \Big) \geq \E \Phi\Big( \sum_{i=1}^n \sigma_i U_i \Big),
\end{equation}
where $U_1,\ldots,U_n$ are arbitrary independent symmetric uniform random variables and $R_1,\ldots,R_n$ are independent positive random variables, independent of $U_i$ satisfying $\E R_i^2=\sigma_i^2$. By Jensen's inequality, \eqref{eq:with-R} is equivalent to the coordinatewise convexity of the function
\begin{equation}
	H(x_1,\ldots,x_n)=  \E \Phi\Big( \sum_{i=1}^n \sqrt{x_i} U_i \Big), \qquad x_1,\ldots,x_n > 0.
\end{equation}
We claim that this is equivalent to the convexity of 
\begin{equation}
	h(x) = h_{U_1, U_2}(x)= \E_{U_1,U_2} \Phi(\sqrt{x} U_1 +U_2), \qquad x > 0,
\end{equation}
where $U_1,U_2$ are arbitrary independent symmetric uniform random variables. Indeed, one direction is clear as $h(x)=H(x,1,0,\ldots,0)$. To prove that $H$ is convex in $x_1$ assuming the convexity of $h$, it suffices to write $S=\sum_{i=2}^n \sqrt{x_i} U_i$ in the form $S = R U_2$ (using Lemmas \ref{lm:unimodmix} and \ref{lm:unimodsum}), where $U_2$ is some uniform symmetric random variable, and $R$ is positive, independent of $U_2$. Then, we have
$$H(x_1,\ldots,x_n) = \E_R \E_{U_1,U_2} \Phi(\sqrt{x_1}U_1 + RU_2) = \E_R\E_{U_1,U_2} h_{U_1,RU_2}(x_1),$$
 which is a mixture of convex functions. As a result, Theorem \ref{thm:Phi} is a consequence of the following elementary observation.

\begin{lemma}\label{lm:Phi}
Let $\Phi:\R\to\R$ be an even function of class $C^3$. Then, the function
\begin{equation}
h(x) = \int_{-b}^b\int_{-a}^a\Phi(u+\sqrt{x}v) \dd u \dd v
\end{equation}
is convex on $[0,+\infty)$ for every $a,b>0$ if and only if $\Phi'''(x)\geq0$ for every $x\in[0,\infty)$.
\end{lemma}

\begin{proof}
Suppose that $\Phi'''(x)\geq0$ for every $x\in[0,\infty)$. To show the convexity of $h$ observe that since
\[
h'(x) = \int_{-b}^b \frac{\Phi(a+\sqrt{x}v)-\Phi(-a+\sqrt{x}v)}{2\sqrt{x}}v\dd v,\]
by a simple rescaling and homogeneity, it is enough to show that the function
\begin{equation}
y \longmapsto \frac{\Phi(y+1)-\Phi(y-1)}{y} = \frac{1}{y}\int_{-1}^1\Phi'(y+t)\dd t
\end{equation}
is nondecreasing on $(0,\infty)$. This follows by the monotonicity of slopes, because the function $y \mapsto \int_{-1}^1 \Phi'(y+t) \dd t $ vanishes at $y=0$ and is convex (as can easily be seen by observing that $\Phi'''$ is odd and distinguishing cases $y>1$ and $y\in(0,1)$).

To show the converse, consider $H_{a,b}(x) = \int_{-b}^b f(a,x,v) \dd v$, where
\begin{align*}
 f(a,x,v) =
\int_{-a}^a \Big(\Phi(u+\sqrt{x}v)-\Phi(u)-\sqrt{x}v\Phi'(u)-\frac{1}{2}xv^2\Phi''(u)-\frac{1}{6}x^{3/2}v^3\Phi'''(u)\Big)
\dd u.
\end{align*}
Since $H_{a,b}(x)$ differs from $h(x)$ by an affine function, $H_{a,b}$ is also convex on $[0,\infty)$. Note that $v \mapsto f(a,x,v)$ is an even function and satisfies
$f(a,x,0) = \frac{\partial}{\partial v}f(a,x,0) = \frac{\partial^2}{\partial v^2}f(a,x,0) = \frac{\partial^3}{\partial v^3}f(a,x,0) = 0$ and $\frac{\partial^4}{\partial v^4}f(a,x,0) = 2x^2\Phi'''(a)$. Therefore, we find that
\[
\lim_{b\to 0^+} \frac{1}{b^5}H_{a,b}(x) =\frac{2}{5}\lim_{b\to0^+} \frac{f(a,x,b)}{b^4} =  \frac{1}{30}x^2\cdot \Phi'''(a)\]
and we know this is a convex function of $x$ on $[0,\infty)$ for every $a \geq 0$ as a pointwise limit of convex functions. Thus, $\Phi'''(a) \geq 0$ for every $a \geq 0$. Changing $\Phi$ to $-\Phi$ proves the opposite statement.
\end{proof}

\begin{remark} \label{rem:regularity}
The proof of Theorem \ref{thm:Phi} shows that a sufficient condition for \eqref{eq:Phi} to hold is that the function $\Phi$ is only of class $C^1$ with $\Phi'$ being convex on $[0,\infty)$. Therefore, choosing $X_i=\sigma_i G_i$ to be Gaussian random variables with variances $\sigma_i^2$ and $\Phi(x)=|x|^p$, $p\geq2$, shows that for every real scalars $\sigma_1,\ldots,\sigma_n$,
\begin{equation}\E\Big|\sum_{i=1}^n \sigma_i U_i\Big|^p \leq \E\Big|\sum_{i=1}^n \sigma_i G_i\Big|^p = \gamma_p^p \Big(\sum_{i=1}^n \sigma_i^2\Big)^{p/2}.
\end{equation}
The same argument also gives the Gaussian optimal constant when $p\in(1,2)$, yet it does not work for $p<1$ due to the lack of the differentiability of $\Phi(x)=|x|^p$ at 0. 
\end{remark}


\section{Moment comparison for symmetric log-concave functions} \label{sec:momentcomparison}

In this section we shall present the proof of Theorem \ref{moments-main}. In Subsection \ref{4.1} we describe some properties of the families $\mc{L}_n^\pm$. We shall need those properties in particular for the proof of Theorem \ref{moments-main}{\it(ii)}. In Subsection 4.2 we formulate and prove two rather standard topological facts concerning Euclidean balls. In Subsection \ref{4.3} we introduce the main ingredients needed for the inductive proof of Theorem \ref{moments-main}\emph{(i)}. We also establish some technical preparatory facts. In Subsection 4.4 we formulate and prove Theorem \ref{A=B=C}, which can be seen as a strengthening of Theorem \ref{moments-main}\emph{(i)} needed for our induction-based argument to work. Finally, we prove Theorem \ref{moments-main}.   


\subsection{Properties of $\mc{L}_n^\pm$}\label{4.1}

The following three elementary lemmas are crucial for the arguments presented in this subsection.

\begin{lemma}\label{l1}
Suppose that $a_1,\ldots, a_n$, $b_1,\ldots,b_n$ are real numbers. Then the function
\begin{equation}
	h(t) = a_1 t^{b_1}+ \cdots + a_n t^{b_n} 
\end{equation}
is either identically zero or it has at most $n-1$ zeroes in the interval $(0,\infty)$. Moreover, if $h$ has exactly $n-1$ zeroes in $(0,\infty)$, then every zero is a sign change point of $h$.
\end{lemma}


\begin{proof}
For the proof of the first statement we proceed by induction on $n$. The statement is trivial for $n=1$. Assume that the assertion is true for some $n-1$ and, without loss of generality, that 
\[
	h(t)=a_1 t^{b_1}+ \cdots  + a_n t^{b_n}
\]
is not of the form $h(t)=at^b$. The equation $h(t)=0$ is equivalent to $\tilde{h}(t)=0$ where 
\[
	\tilde{h}(t)= a_1 + a_{2} t^{b_2-b_1}+ \cdots  + a_{n} t^{b_{n}-b_1}
\]
is non-constant. To prove our assertion by contradiction, suppose that the latter has more than $n-1$ solutions in $(0,\infty)$. Then, Rolle's theorem shows that the function
\[
	\tilde{h}'(t) = (b_2-b_1) a_{2} t^{b_2-b_1-1}+ \cdots + (b_n-b_1) a_n t^{b_n-b_1-1},
\]
which is not identically zero, has at least $n-1$ zeros. This contradicts the inductive hypothesis. 

For the second part let us assume, by contradiction, that there is a point $t_\star>0$ such that $h(t_\star)=0$, but $t_\star$ is a local extremum for $h$. In particular, the function $h$ is not of the form $h(t)=at^b$. Then, the function $\tilde{h}$ defined above has exactly $n-1$ zeroes in $(0,\infty)$ and $t_\star$ is a local extremum of $\tilde{h}$. Therefore, by Rolle's theorem $\tilde{h}'$ has $n-2$ zeroes lying strictly between the zeroes of $\tilde{h}$ and additional one at $t_\star$. This means that $\tilde{h}'$
has at least $n-1$ zeroes in $(0,\infty)$, which contradicts the first part of the lemma.
\end{proof}

The formulation of the next lemma appeared as Problem 76 in \cite{PSz}. We include its proof for completeness.  

\begin{lemma} \label{l2}
For any real numbers $p_1 < p_2 \ldots < p_n$ and $0<t_1 < t_2 \ldots < t_n$ the determinant of the matrix $A=\left(t_i^{p_j}\right)_{i,j=1}^n$ is positive. 
\end{lemma}

\begin{proof} 
We first show that $\det(A) \ne 0$. To prove it by contradiction, assume that the matrix $A$ is singular and take a non-zero vector $c=(c_1,\ldots,c_n)$ such that $Ac=0$. Thus, if $f$ is given by
\begin{equation}
	f(t) = \sum_{j=1}^n c_j t^{p_j}, \qquad t>0,
\end{equation}
we have $f(t_i)=0$ for every $i=1,\ldots,n$. Since some of the $c_i$ are non-zero, the function $f$ is not identically zero,  which contradicts Lemma \ref{l1}. 

To prove that the sign of $\det(A)$ is positive we proceed by induction. The assertion is clear for $n=1$. From the first part we deduce that the function 
\[
	(t_{n-1},\infty) \ni t_n \mapsto \det\left( \left(t_i^{p_j}\right)_{i,j=1}^n \right)
\]
has constant sign. It therefore suffices to check the sign in the limit $t_n \to \infty$. Expanding the determinant with respect to the last row we get
\[
	\lim_{t_n \to \infty} \frac{1}{t_n^{p_n}} \det\left( \left(t_i^{p_j}\right)_{i,j=1}^n \right) = \det(\left(t_i^{p_j}\right)_{i,j=1}^{n-1}),
\]
which is positive by induction hypothesis. This completes the proof.
\end{proof}

For $n\geq1$ let us define the moment map $\Psi_n: \{f: f \geq 0\} \to [0,\infty]^n$ given by
\begin{equation} \label{eq:Psi}
	\Psi_n(f) = (m_1(f),\ldots, m_{n}(f)), \qquad \textrm{where} \  m_i(f) = \int_{0}^\infty t^{p_i}f(t) \dd t,
\end{equation}
for every $i\in\{1,\ldots,n\}$.

\begin{lemma}\label{l3}
Suppose that $f,g:[0,\infty)\to\R_+$ are two measurable functions such that $f-g$ changes sign at most $n-1$ times on $(0,\infty)$. If $\Psi_n(f)=\Psi_n(g)$, then $f=g$ a.e.
\end{lemma}

\begin{proof}
Suppose that $f-g$ changes sign at some points $0<t_1<t_2<\ldots<t_{k}$, where $k \leq n-1$. For real numbers $c_1,\ldots,c_{k}$ consider the function
\begin{equation}
	h(t) = t^{p_{k+1}} + \sum_{i=1}^{k} c_i t^{p_i}.
\end{equation}
Using Lemma \ref{l2}, we see that it is possible to find $c_1,\ldots,c_{k}\in\R$ such that $h(t_i)=0$ for every $i=1,\ldots,k$ (since this involves solving a linear system of equations whose determinant is non-zero). From Lemma \ref{l1}, for this choice of $c_1,\ldots,c_{k}$, the function $h$ has exactly $k$ roots in $(0,\infty)$ and each root corresponds to a sign change of $h$. Therefore, the function $h(f-g)$ has a fixed sign. However, since $\Psi_n(f)=\Psi_n(g)$ implies $\Psi_{k}(f)=\Psi_k(g)$, we get $\int_0^\infty h(f-g) = 0$, and thus $f=g$ a.e.
\end{proof}

We begin our study of the families $\mc{L}_n^\pm$ with a lemma which will be needed to show the uniqueness
in Theorem \ref{moments-main}{\it(i)}.

\begin{lemma}\label{injectivity}
The map $\Psi_n$ is injective on $\mc{L}_n^\pm$. 
\end{lemma}

\begin{proof}
A careful case analysis shows that if $f,g\in\mc{L}_n^+$ or $f,g\in\mc{L}_n^-$, then $f-g$ changes sign at most $n-1$ times on $(0,\infty)$. Therefore, Lemma \ref{l3} shows that if $\Psi_n(f)=\Psi_n(g)$, then $f=g$ a.e. It follows that $f=g$ everywhere, due to the convention $\infty \cdot 0 = 0$ which leads to the lower semi-continuity of the members of $\mc{L}_n^\pm$.
\end{proof}

We are ready to formulate and prove our main proposition of this subsection.

\begin{proposition}\label{prop:extremalmoments}
For $n\geq1$, suppose that the functions $f \in \mc{L}$, $f_+ \in \mc{L}_{n}^+$ and $f_- \in \mc{L}_{n}^-$ are such that 
\begin{equation}
\Psi_n(f_+)=\Psi_n(f_-)=\Psi_n(f)
\end{equation}
and let $p_{i_1}<p_{i_2}<\ldots<p_{i_{n+1}}$ be the increasing rearrangement of $p_1,\ldots,p_{n+1}$. The following hold true.
\begin{itemize}
\item[{\it(i)}] If $n+1=i_k$ and $n+1-k$ is even, then 
\begin{equation} 
m_{n+1}(f_-) \leq m_{n+1}(f) \leq m_{n+1}(f_+).
\end{equation}
If $n+1-k$ is odd, then the above inequalities are reversed.
\item[{\it(ii)}] If $f_+$ or $f_-$ belongs to $\mc{L}_{n-1}^+ \cup \mc{L}_{n-1}^-$ then $f_+=f_-$ and, in particular,  $m_{n+1}(f_-)=m_{n+1}(f_+)$.  
\item[{\it(iii)}] If $f_+ \notin \mc{L}_{n-1}^+\cup\mc{L}_{n-1}^-$ and $f_- \notin \mc{L}_{n-1}^+\cup\mc{L}_{n-1}^-$ then $m_{n+1}(f_-) \ne m_{n+1}(f_+)$.
\item[{\it(iv)}] If $m_{n+1}(f)=m_{n+1}(f_\pm)$ then $f=f_\pm$ a.e.
\end{itemize}
\end{proposition}

\begin{proof}
We shall prove that if $n+1-k$ is even, then $m_{n+1}(f_-) \leq m_{n+1}(f)$ and the reverse holds if $n+1-k$ is odd. The inequalities for $f_+$ are identical. We can clearly assume that $f$ is not equal to $f_-$. Then, by the log-concavity of $f$ and the definition of $\mc{L}_n^-$, the function $f-f_-$ changes sign at most $n$ times on $(0,\infty)$. Combining this fact with the assumption $\Psi_n(f)=\Psi_n(f_-)$ and Lemma \ref{l3}, we infer that $f-f_-$ changes sign exactly $n$ times on $(0,\infty)$. As in the proof of Lemma \ref{l3}, take $h(t) = \sum_{i=1}^{n+1} c_i t^{p_i}$ with $c_{n+1}=1$ and choose $c_1,\ldots,c_{n}\in\R$ such that $h(f-f_-)$ has a fixed sign. Note that in a small neighbourhood to the right of the last sign change (when $f_-$ jumps to $0$) the sign of  $f-f_-$ must be positive, since otherwise the number of sign changes would be strictly less than $n$. What remains is to examine the sign of the function $h$ to the right of the last sign change or, equivalently, the sign of the coefficient $c_s$, where $s=i_{n+1}$ is the index of the maximal exponent $p_{i_{n+1}}$. We can clearly assume that $p_1<\ldots<p_n$, therefore $s=n$ or $s=n+1$. If $s=n+1$ we have $c_s=1$ and thus $h(f-f_-)\geq0$. In this case we get  
\[
	\int_{0}^\infty t^{p_{n+1}}(f(t)-f_-(t)) \dd t = \int_{0}^\infty h(t)\big(f(t)-f_-(t)\big)\dd t  \geq  0.
\]
Assume now that $s=n$, and recall that the vector $c=(c_1,\ldots,c_n)$ was constructed as the solution to the linear system
\begin{equation}
\underbrace{\begin{pmatrix} 
    t_1^{p_1} & \dots & t_1^{p_n} \\
    \vdots & \ddots & \vdots \\
    t_n^{p_1} &    \dots    & t_n^{p_n}
    \end{pmatrix}}_\text{$A$} \cdot
\begin{pmatrix}
c_1 \\ 
\vdots \\
c_n
\end{pmatrix}
= - \begin{pmatrix}
t_1^{p_{n+1}} \\
\vdots \\
t_n^{p_{n+1}}
\end{pmatrix},
\end{equation}
where $\det(A)>0$ from Lemma \ref{l2}. Hence, a straightforward application of Cramer's rule, shows that $c_n$ has the same sign as
\begin{equation}
- \det \begin{pmatrix} 
    t_1^{p_1} & \dots & t_1^{p_{n-1}} & t_1^{p_{n+1}}\\
    \vdots & \ddots & \vdots & \vdots\\
    t_n^{p_1} &    \dots    & t_n^{p_{n-1}} & t_n^{p_{n+1}}
    \end{pmatrix},
\end{equation}
which is positive if $n+1-k$ is even and negative if $n+1-k$ is odd, as can be seen by repeatedly swapping columns so that the exponents $p_i$ are ordered and then applying Lemma \ref{l2}. Knowing the sign of $c_n$, we then find $\lim_{t\to\infty}h(t)$ as before and thus decide whether $h(f-f_-)$ is nonnegative or nonpositive. Then {\it(i)} follows by integrating.

Part {\it(ii)} is an immediate consequence of Lemma \ref{l3}, since if, say $f_+ \in\mc{L}_{n-1}^+\cup\mc{L}_{n-1}^-$, then for any $f \in \mc{L}$ the function $f_+-f$ changes sign at most $n-1$ times, in particular so does $f_+ - f_-$. To prove part {\it(iii)}, first observe that the assumption implies that $f_+$ is not equal to $f_-$. Thus, the same argument used for {\it(i)} shows that $f_+-f_-$ changes sign exactly $n$ times and choosing the function $h$ as above, gives $\int_0^\infty h(f_+-f_-) \ne 0$, since $h(f_+-f_-)$ is not identically zero and has a fixed sign. Part {\it(iv)} follows again from Lemma \ref{l3} by observing that $f-f_\pm$ changes sign in at most $n$ points and $\Psi_{n+1}(f)=\Psi_{n+1}(f_\pm)$.
\end{proof}


\subsection{Topological facts}\label{4.2} We will also need the following standard topological lemmas.

\begin{lemma}\label{top1} 
Let $B_0\subseteq\R^n$ be a set homeomorphic to the closed Euclidean ball $B_2^n$ and suppose that $F_+, F_-:B_0\to\R$ are two continuous functions such that $F_+(x) \geq F_-(x)$ for every $x\in B_0$, with equality if and only if $x \in \partial B_0$. Then, the set
\begin{equation} \label{eq:Cdefinition}
	C = \{(x,y) \in B_0 \times \R: \ F_-(x) \leq y \leq F_+(x)  \}
\end{equation}
is homeomorphic to the closed Euclidean ball $B_2^{n+1}$ and
\begin{equation} \label{eq:Cboundary}
	\partial C = \big\{(x,F_-(x)): \ x \in B_0 \big\} \cup \big\{(x,F_+(x)): \ x \in B_0 \big\}.
\end{equation}
\end{lemma}

\begin{proof}
Let $h:B_2^n \to B_0$ be a homeomorphism. By considering the functions $F_+ \circ h$ and $F_- \circ h$ on $B_2^n$, we can clearly assume that $B_0=B_2^n$. Then, we claim that the function $\Omega(x,y) = (x,\omega(x,y))$, where
\begin{equation} \label{eq:homeomorphism}
	\omega(x,y)= \left\{  \begin{array}{ll}  F_+(x) + \frac{F_+(x)-F_-(x)}{2} \Big( \frac{y}{(1-\|x\|_2^2)^{1/2}}-1 \Big), \quad  & \|x\|_2<1 \\
	F_+(x)=F_-(x), \quad & \|x\|_2=1 \end{array} \right.,
\end{equation}
is a continuous map from $B_2^{n+1}$ to $C$. Indeed, the continuity on the interior of $B_2^{n+1}$, as well as the continuity at points $(x,y)\in\partial B_2^{n+1}$ with $y\neq0$, is clear. We are left with checking the continuity at points $(x,0)$, where $x$ satisfies $\|x\|_2=1$. Suppose $(x_n, y_n) \to (x,0)$. It is enough to show that $\omega(x_n,y_n) \to \omega(x,0)=F_+(x)=F_-(x)$. We have $\omega(x_n,y_n) \in [F_-(x_n),F_+(x_n)]$ and the desired convergence follows by the sandwich rule.

Moreover, the inverse of the map \eqref{eq:homeomorphism} is given by $\Omega^{-1}(x,y)=(x,\theta(x,y))$, where
\begin{equation} \label{eq:inversehomeomorphism}
	\theta(x,y) = \left\{  \begin{array}{ll} \left(\frac{2}{F_+(x)-F_-(x)}(y-F_+(x))+1\right)(1-\|x\|_2^2)^{1/2} , \quad &  \|x\|_2<1 \\
	0, \quad &  \|x\|_2=1 \end{array} \right.
\end{equation}
and is also continuous. Indeed the only problematic case in checking the continuity occurs when $F_+(x)=F_-(x)$, that is, $\|x\|_2=1$. In this case, if $(x_n,y_n) \to (x,0)$ then $\theta(x_n,y_n) \to \theta(x,0)=0$ since
\[
\theta(x_n,y_n) \in \left[-(1-\|x_n\|_2^2)^{1/2}, (1-\|x_n\|_2^2)^{1/2}\right],
\]
and we can again use the sandwich rule. Hence $C$ is indeed homeomorphic to $B_2^{n+1}$. The description of the boundary of $C$ follows from the continuity of $F_+$ and $F_-$ and from the fact that they coincide on the boundary of $B_2^n$.
\end{proof}

\begin{lemma}\label{top2}
Let $P$ and $C$ be two subsets of $\mathbb{R}^n$ homeomorphic to a closed Euclidean ball $B$. Consider a continuous function $f:P\to\R^n$ that is injective on $\mathrm{int}(P)$ and assume that $f(P)\subseteq C$ and $f(\partial P)=\partial C$. Then $f(P)=C$.
\end{lemma}

\begin{proof}
We can clearly assume that $P=C=B$. Suppose the assertion does not hold, that is, there exists $y_0 \in B$ such that $y_0 \notin f(B)$. For any $\theta \in S^{n-1}$ let us define
\[
	r(\theta) = \mbox{the point } y \ \textrm{in} \ \{y_0+t\theta: \ t\geq0\}\cap f(B) \ \mbox{which is closest to } y_0.
\] 
Since $f(B)$ is compact, $r(\theta)$ is well defined. We claim that $r(\theta) \notin f(\textrm{int}(B))$. Indeed, by the invariance of domain theorem (see \cite[Theorem~2B.3]{Hat}), $f|_{\textrm{int}(B)}$ is an open map and therefore $f(\textrm{int}(B))$ is an open subset of $\R^n$. If $r(\theta)$ was in $f(\textrm{int}(B))$, then it would be contained in $f(\textrm{int}(B))$
along with a ball around it, hence contradicting its minimality. We get that $r(\theta) \in \partial B$ for any $\theta \in S^{n-1}$ and thus $f(B) \subseteq \partial B$. In particular $f(\textrm{int}(B)) \subseteq \partial B$, which is a contradiction since $f(\textrm{int}(B))$ is open.    
\end{proof}

\subsection{Technical facts}\label{4.3}

For every function space $\mc{L}_n^\pm$ we denote by $\mc{P}_n^\pm \subset [0,\infty]^n$ the corresponding parameter space of the vectors of parameters $(a,b)$ appearing in \eqref{eq:Lndefine}. The parameter space $\mc{P}_n^{\pm}$ is compact (in the usual topology of $[0,\infty]^n$) and homeomorphic to the closed Euclidean ball $B_2^n$. These parameter spaces give rise to natural maps $e_n^\pm:\mc{P}_n^\pm \to \mc{L}_n^\pm$, which are injective on the interiors of $\mc{P}_n^\pm$ (but not on the boundaries). A simple case analysis also shows that
\begin{equation} \label{eq:enboundary}
e_n^\pm (\partial \mc{P}_n^\pm) = \mc{L}_{n-1}^+ \cup \mc{L}_{n-1}^-.
\end{equation}
Fix $n\geq1$ and distinct $p_1,\ldots,p_{n+1}\in(-1,\infty)$. For $M>0$ consider the class
\begin{equation} \label{eq:lnm}
	\mc{L}_{n,M}^\pm = \Big\{ f\in\mc{L}_n^\pm: \ \int_0^\infty f(t)\dd t \leq M\Big\}.
\end{equation}
and note that $\bigcup_{M>0}\mc{L}_{n,M}^\pm = \mc{L}_{n}^\pm \setminus \{f \equiv 1\}$. Denote by $\mc{P}_{n,M}^\pm= (e_n^\pm)^{-1}(\mc{L}_{n,M}^\pm)$ the corresponding parameter space. Moreover, if $p=\min_{i=1,\ldots,n+1} p_i$ and $P=\max_{i=1,\ldots,n+1} p_i$, we equip the space $\mc{L}_{n,M}^\pm$ with the metric
\begin{equation} \label{eq:dmetric}
	d(f,g) = \int_{0}^\infty |f(t)-g(t)| (t^{p}+t^P) \dd t,
\end{equation}
which is well defined since the only log-concave function $f\in\mc{L}$ which does not decay exponentially is $f \equiv 1$. 

We will prove the following technical proposition. 

\begin{proposition}\label{prop:technical}
For every $n \geq 1$ and $M > 0$ the following hold true. 
\begin{itemize}
	\item[\it (i)] The functionals $m_{i}$ are continuous  on $(\mc{L}_{n,M}^\pm,d)$ for every $i=1,\ldots,n+1$. As a consequence, the map $\Psi_n$ is also   continuous  on $(\mc{L}_{n,M}^\pm,d)$.
	\item[\it (ii)]  The natural map $e_n^\pm:\mc{P}_{n,M}^\pm \to \mc{L}_{n,M}^\pm$ is a continuous map between compact spaces. 
	\item[\it (iii)]  The map $\Psi_n \circ e_n^\pm:\mc{P}_{n}^\pm \to \R^n$ is continuous.
	\item[\it (iv)] The map $\Psi_n:\mc{L}_{n,M}^\pm\to\Psi_n(\mc{L}_{n,M}^\pm)$ is a homeomorphism.
	\item[\it (v)] The map $m_{n+1}\circ(\Psi_n)^{-1}:\Psi_n(\mc{L}_n^\pm)\to\R_+\cup\{\infty\}$ is continuous.
\end{itemize}
\end{proposition}

\begin{proof}
\textit{(i)} Since $t^{p_i} \leq t^p+t^P$ for any $i=1,\ldots,n+1$, the continuity of $m_i$ is evident.

\textit{(ii)} Suppose that $(a^{(k)}, b^{(k)}) \in \mc{P}_{n,M}^\pm$ satisfy $(a^{(k)},b^{(k)}) \to (a,b)$ for some $(a,b) \in \mc{P}_{n,M}^\pm$. Let $f_k=e_n^\pm(a^{(k)}, b^{(k)})$ and  $f= e_n^\pm(a,b)$. Then $f_k \to f$ a.e. Indeed, the only point $t$ where $a_i^{(k)}(t-b_i^{(k)})$ might not converge to $a_i(t-b_i)$ is $t=b_i$, when $b_i$ is finite. Therefore, the convergence holds everywhere except for finitely many points. For every function $g \in \mc{L}_{n,M}^\pm$ we have $2M g(2M) \leq \int_0^\infty g \leq M$. This gives $g(2M) \leq 1/2$ and by log-concavity $g(t) \leq g(2M)^{t/2M} \leq 2^{-t/2M}$ for $t \geq 2M$. Thus, $g(t) \leq 2^{-t/2M}\1_{\{t \geq 2M\}} + \1_{\{t < 2M\}}$. We therefore get 
\[
|f_k(t)-f(t)| \leq 2 \cdot 2^{-t/2M}\1_{\{t \geq 2M\}} + 2\1_{\{t < 2M\}}
\]
and thus $\int_0^\infty|f_k(t)-f(t)|(t^p+t^P) \dd t \to 0$ by Lebesgue's dominated convergence theorem, i.e. $d(f_k,f) \to 0$. Hence, $e_n^\pm:\mc{P}_{n,M}^\pm \to \mc{L}_{n,M}^\pm$ is a continuous map. Since $\mc{L}_{n,M}^\pm$ is a closed subset of $\mc{L}_n^\pm$, $\mc{P}_{n,M}^\pm$ is a closed subset of the compact space $\mc{P}_n^\pm$, and thus it is compact. As a result, $\mc{L}_{n,M}^\pm = e_n^\pm(\mc{P}_{n,M}^\pm)$ is also compact.

\textit{(iii)} Let us consider a sequence of parameters $(a^{(k)}, b^{(k)}) \in \mc{P}_{n}^\pm$ converging to $(a,b)\in \mc{P}_{n}^\pm$. If $f=e_n^\pm(a,b)$ is not identically equal $1$, then by a.e.\! convergence of $f_k=e_n^\pm(a^{(k)}, b^{(k)})$ to $f$ we deduce that there exists $L>0$ such that eventually $f_k(L)<1/2$. By the same reasoning as in the proof of part \textit{(ii)} we see that eventually $f_k$ are exponentially bounded on $[L,\infty)$, namely $f_k(t) \leq 2^{-t/L}\1_{\{t \geq L\}} + \1_{\{t < L\}}$. Thus, eventually $f_k \in \mc{L}_{n,M_0}^\pm$ with $M_0=L(1+ 1/2\ln 2)$. Thus, in this case our assertion follows by combining \textit{(i)} and \textit{(ii)}. If $f \equiv 1$ then by Fatou's lemma
\[
	\infty = \int_0^\infty f(t)t^{p_i} \dd t \leq  \liminf_{k \to \infty} \int_0^\infty f_k(t) t^{p_i} \dd t, \qquad i=1,\ldots,n,
\]
and thus $\Psi_n(f_k) \to \Psi_n(f)=(\infty,\ldots,\infty)$.

\textit{(iv)} By Lemma \ref{injectivity} the map $\Psi_n$ is injective. From point \textit{(i)} it is also continuous. Hence, $\Psi_n:\mc{L}_{n,M}^\pm\to\Psi_n(\mc{L}_{n,M}^\pm)$ is a continuous bijection defined on the compact space $\mc{L}_{n,M}^\pm$ with values in the Hausdorff space $[0,\infty]^n$. Consequently, it is a homeomorphism.

{\it(v)} To prove the continuity at a point $\Psi_n(f) = m \in \Psi_n(\mc{L}_n^\pm)$ which is not $(\infty,\ldots,\infty)$, take a sequence $m_k = \Psi_n(f_k)$ convergent to $m$. It suffices to show that eventually all $f_k$ belong to $\mc{L}_{n,M_1}^\pm$ for some $M_1$ because {\it(i)} and {\it(iv)} immediately imply that for every $M>0$, $m_{n+1}\circ(\Psi_n)^{-1}$ restricted on $\Psi_n(\mc{L}_{n,M}^\pm)$ is continuous. For any $f \in \mc{L}$ and any $p,q>-1$  we have
\begin{equation}\label{moment comparison}
	\left( \int_{0}^\infty f(t) t^p \dd t \right)^{\frac{1}{p+1}} \leq C_{p,q} \left( \int_{0}^\infty f(t) t^q \dd t \right)^{\frac{1}{q+1}},
\end{equation}
where
\[
	C_{p,q}=\max\left\{\frac{(q+1)^{\frac{1}{q+1}}}{(p+1)^{\frac{1}{p+1}}}, \frac{\Gamma(p+1)^{\frac{1}{p+1}}}{\Gamma(q+1)^{\frac{1}{q+1}}}\right\}.
\]
To prove the above inequality choose unique functions $f_+ \in \mc{L}_1^+$ and $f_- \in \mc{L}_1^-$ such that $\int_0^\infty f(t) t^q\dd t = \int_0^\infty f_+(t) t^q\dd t = \int_0^\infty f_-(t) t^q\dd t$. Applying Proposition \ref{prop:extremalmoments} in the case $n=1$ with $p_1=q$ and $p_2=p$ reduces proving \eqref{moment comparison} to the case $f \in \{f_+, f_-\}$. The inequality follows by computing the resulting constants in these two cases. Since $\Psi_n(f_k)$ converges, there is $M_0>0$ such that $m_1(f_k) \leq M_0$ for any $k \geq 1$. It follows that $\int_0^\infty f_k \leq C_{0,p_1} M_0^{1/(p_1+1)}$ and so we can take $M_1=C_{0,p_1} M_0^{1/(p_1+1)} +1$. To prove the continuity at $(\infty,\ldots,\infty)$ is suffices to observe that due to \eqref{moment comparison} we get that $\int_0^\infty f_k(t)t^{p_1}\dd t \to \infty$ implies $\int_0^\infty f_k(t)t^{p_{n+1}}\dd t \to \infty$.  
\end{proof}

\subsection{Proof of Theorem \ref{moments-main}}\label{4.4}

Define $\mc{A}_n^\pm = \Psi_n(\mc{L}_n^\pm)$ and $\mc{B}_n=\Psi_n(\mc{L})$. To establish Theorem \ref{moments-main}\textit{(i)} we shall prove that $\mc{A}_n^\pm=\mc{B}_n$. Consider the functions $F_+$ and $F_-$ on $\mc{B}_{n-1}$, given by
\begin{equation} \label{eq:F+}
F_+(m_1,\ldots,m_{n-1}) = \sup\big\{m_n(f): \ f\in\mc{L} \mbox{ and } m_i(f)=m_i, \ i=1,\ldots,n-1\big\}  
\end{equation}
and
\begin{equation} \label{eq:F-}
F_-(m_1,\ldots,m_{n-1}) = \inf\big\{m_n(f): \ f\in\mc{L} \mbox{ and } m_i(f)=m_i, \ i=1,\ldots,n-1\big\}
\end{equation}
and let
\begin{equation} \label{eq:Cn}
\mc{C}_n=\big\{ (m_1,\ldots,m_n)\in\mc{B}_{n-1}\times\R: \ m_n\in\big[F_-(m_1,\ldots,m_{n-1}),F_+(m_1,\ldots,m_{n-1})\big]\big\}.
\end{equation}
It is clear from the definition of these sets that
\begin{equation}
\mc{A}_n^\pm \subseteq \mc{B}_n \subseteq \mc{C}_n.
\end{equation}
We will prove the following strengthening of Theorem \ref{moments-main}\textit{(i)}.

\begin{theorem} \label{A=B=C}
For every $n\geq1$ we have $\mc{A}_n^+=\mc{A}_n^- = \mc{B}_n = \mc{C}_n$. Moreover, these sets are homeomorphic to the Euclidean ball $B_2^n$ and their boundary is $\Psi_n(\mc{L}_{n-1}^+)\cup\Psi_n(\mc{L}_{n-1}^-)$.
\end{theorem}

\begin{proof} The proof goes by induction on $n$. For $n=1$ we get 
\begin{equation}
\mc{L}_1^+=\{f(t)=e^{-at}: \ 0\leq a\leq\infty\} \ \ \mbox{and} \ \ \mc{L}_1^-=\{f(t)=\1_{[0,b]}: \ 0\leq b\leq\infty\}.
\end{equation} 
Thus, since $\mc{P}_1^\pm=[0,\infty]$, we get $m_1(e_1^+(a))=a^{-(p_1+1)}\Gamma(p_1+1)$ and $m_1(e_1^-(b))=\frac{1}{p_1+1}b^{p_1+1}$, which implies that $\mc{A}_1^\pm = \mc{B}_1=\mc{C}_1=[0,\infty]$. Since $\mc{L}_0^\pm=\{\1_{\{0\}},1\}$ we get $\partial \mc{B}_1=\{0,\infty\}=\Psi_1(\mc{L}_0^+)=\Psi_1(\mc{L}_0^-)$. Therefore, the assertion is true for $n=1$. 

Suppose that the claim is true for $n$ constraints and we will show it for $n+1$. We will first determine the boundary of $\mc{C}_{n+1}$. Let $m=(m_1,\ldots,m_{n})\in \mc{B}_{n}$. By the induction hypothesis there exist $f_+ \in \mc{L}_{n}^+$ and $f_- \in \mc{L}_{n}^-$ such that 
$$\Psi_n(f_+)=\Psi_n(f_-)=(m_1,\ldots,m_n).$$
and then Proposition \ref{prop:extremalmoments}{\it(i)} shows that for any $f \in \mc{L}$ such that $\Psi_n(f)=(m_1,\ldots,m_n)$, we have 
\begin{equation} \label{eq:usethepropo}
\min\big\{m_{n+1}(f_-),m_{n+1}(f_+)\big\}\leq m_{n+1}(f) \leq \max\big\{m_{n+1}(f_-),m_{n+1}(f_+)\big\},
\end{equation} 
depending on the sequence $p_1,\ldots,p_{n+1}$. Consider the functions
\begin{equation} \label{eq:foundF-}
\widetilde{F}_-=\min\big\{ m_{n+1} \circ (\Psi_{n}|_{\mc{L}_n^-})^{-1}, m_{n+1} \circ (\Psi_{n}|_{\mc{L}_n^+})^{-1} \big\}
\end{equation}
and 
\begin{equation} \label{eq:foundF+}
\widetilde{F}_+ =\max\big\{ m_{n+1} \circ (\Psi_{n}|_{\mc{L}_n^-})^{-1}, m_{n+1} \circ (\Psi_{n}|_{\mc{L}_n^+})^{-1} \big\}.
\end{equation}
A combination of the induction hypothesis $\mc{B}_{n}=\mc{A}_n^\pm=\Psi_{n} (\mc{L}_{n}^\pm)$ with Proposition \ref{prop:technical}{\it(v)} yields the continuity of $\widetilde{F}_+$ and $\widetilde{F}_-$ on $\mc{B}_n$, which moreover is a set homeomorphic to $B_2^n$.
It also follows from the induction hypothesis that the boundary of $\mc{B}_{n}$ is $\Psi_{n}(\mc{L}_{n-1}^+) \cup \Psi_{n}(\mc{L}_{n-1}^-)$.
We would like to show that pointwise $\widetilde{F}_+ \geq \widetilde{F}_-$ with equality only on the boundary of $\mc{B}_n$.
Indeed, take a point $m \in \mc{B}_n$ and unique (by Lemma \ref{injectivity}) functions $f_\pm \in \mc{L}_n^\pm$ such that $m = \Psi_n(f_-) = \Psi_n(f_+)$. If $m \in \partial \mc{B}_n = \Psi_{n}(\mc{L}_{n-1}^+) \cup \Psi_{n}(\mc{L}_{n-1}^-)$, then $f_\pm \in \mc{L}_{n-1}^+ \cup \mc{L}_{n-1}^-$, so $f_- = f_+$ and $\widetilde{F}_-(m)=\widetilde{F}_+(m)$ by Proposition \ref{prop:extremalmoments}{\it(ii)}. If $m$ is not in $\partial \mc{B}_n$, then neither $f_-$ nor $f_+$ is in $\mc{L}_{n-1}^+ \cup \mc{L}_{n-1}^-$, so by Proposition \ref{prop:extremalmoments}{\it(iii)}, $\widetilde{F}_-(m) < \widetilde{F}_+(m)$.
Combining all the above with Lemma \ref{top1}, we finally infer that the set
\[
	\mc{\widetilde{C}}_{n+1} = \big\{(x,y) \in \mc{B}_n \times \R: \ \widetilde{F}_-(x) \leq y \leq \widetilde{F}_+(x)  \big\} 
\]
is homeomorphic to $B_2^{n+1}$ and that
\begin{equation} \label{eq:foundCn+1}
\partial \mc{\widetilde{C}}_{n+1} = \big\{(x,\widetilde{F}_-(x)): \ x\in\mc{B}_n\big\}\cup\big\{(x,\widetilde{F}_+(x)): \ x\in\mc{B}_n\big\} = \Psi_{n+1}(\mc{L}_{n}^+)\cup\Psi_{n+1}(\mc{L}_{n}^-).
\end{equation}
Moreover, using the notation of \eqref{eq:F+} and \eqref{eq:F-}, we can rewrite \eqref{eq:usethepropo} as $\widetilde{F}_{\pm}=F_{\pm}$, which in turn shows that  $\mc{\widetilde{C}}_{n+1} = \mc{C}_{n+1}$. Therefore, we deduce that
\begin{equation} \label{eq:conclusionCn+1}
\partial \mc{C}_{n+1} = \Psi_{n+1}(\mc{L}_{n}^+)\cup\Psi_{n+1}(\mc{L}_{n}^-).
\end{equation}
The proof will be complete once we show that $\mc{A}_{n+1}^\pm = \mc{C}_{n+1}$. To this end, consider the function $f_{n+1}^\pm:\mc{P}_{n+1}^\pm \to[0,\infty]^{n+1}$, given by $f_{n+1}^\pm = \Psi_{n+1}\circ e_{n+1}^\pm$. It is continuous by Proposition \ref{prop:technical}{\it (iii)} and satisfies 
\[
f_{n+1}^\pm(\mc{P}_{n+1}^\pm) = \Psi_{n+1}(e_{n+1}^\pm(\mc{P}_{n+1}^\pm)) = \Psi_{n+1}(\mc{L}_{n+1}^\pm) = \mc{A}_{n+1} \subseteq \mc{C}_{n+1}
\]
and by \eqref{eq:enboundary} and \eqref{eq:conclusionCn+1},
\[
f_{n+1}^\pm(\partial\mc{P}_{n+1}^\pm) = \Psi_{n+1}(e_{n+1}^\pm(\partial \mc{P}_{n+1}^\pm)) =\Psi_{n+1}(\mc{L}_n^+)\cup\Psi_{n+1}(\mc{L}_n^-) = \partial \mc{C}_{n+1}.
\]
Notice that $f_{n+1}^\pm$ is injective on $\mathrm{int}(\mc{P}_{n+1}^\pm)$, since $\Psi_{n+1}$ is injective on $\mc{L}_{n+1}^\pm$ (by Lemma \ref{injectivity}) and  $e_{n+1}^\pm$ is injective on $\mathrm{int}(\mc{P}_{n+1}^\pm)$. Therefore, since both $\mc{P}_{n+1}^\pm$ and $\mc{C}_{n+1}^\pm$ are homeomorphic to $B_2^{n+1}$, Lemma \ref{top2} gives that $\mc{A}_{n+1}^\pm=f_{n+1}^\pm(\mc{P}^\pm_{n+1})=\mc{C}_{n+1}$, thus completing the proof.
\end{proof}

\begin{remark}
The equality $\mc{B}_n = \mc{C}_n$ provides a structural property of the set $\mc{B}_n$. Namely, its intersection with every line of the form $\{y_0+t e_i, \ t \in \R\}$, where $(e_i)_{i=1}^n$ is the standard basis of $\R^n$, is either a line segment or the empty set.  
\end{remark}

\begin{proof}[Proof of Theorem \ref{moments-main}] 
The existence part of {\it (i)} follows from Theorem \ref{A=B=C}, whereas the uniqueness follows from Lemma \ref{injectivity}. Part {\it (ii)} is a consequence of Proposition \ref{prop:extremalmoments}{\it (i)} and Proposition \ref{prop:extremalmoments}{\it (iv)}. 
\end{proof}

\end{document}